\documentclass[10pt]{article}
\usepackage{amssymb}
\usepackage{amsmath}
\usepackage{amsfonts}
\usepackage{amsthm}
\usepackage{graphics,graphicx}
\usepackage{float}
\usepackage{mathtools}
\usepackage[dvipsnames]{xcolor}
\usepackage{url}

\usepackage[english]{babel}
\usepackage{authblk}

\newtheorem{theorem}{Theorem}
\newtheorem{lemma}{Lemma}
\newtheorem{proposition}{Proposition}
\newtheorem{corollary}{Corollary}

\title{An upper bound on the number of relevant variables in a bounded degree Boolean function on the Hamming graph}
\author[1]{Alexandr Valyuzhenich\thanks{Sobolev Institute of Mathematics, Ak. Koptyug av. 4, Novosibirsk 630090, Russia; Email address: graphkiper@mail.ru}}

\date{}
\begin{document}
\maketitle

\begin{abstract}
In this work, we prove that any Boolean function of degree $d$ on $\mathbb{Z}_{q}^n$, $q\geq 3$, has at most $m_qq^d$ relevant variables,
where $m_q=\frac{2q(q^2+4q+1)}{(q-1)^4}$.
For $q\in \{3,4,5,6,7\}$, we improve this bound to $2.854\cdot 3^d$, $1.749\cdot 4^d$, $1.263\cdot 5^d$, $0.994\cdot 6^d$, and $0.814\cdot 7^d$, respectively.
\end{abstract}

\section{Introduction}
A function $f:\mathbb{Z}_{q}^n\rightarrow \{0,1\}$ is called {\em Boolean}.
Given a function $f$ on $\mathbb{Z}_{q}^n$, a variable $x_i$, $1\leq i\leq n$, is called {\em relevant} (or {\em essential})
if there are two tuples $x,y\in \mathbb{Z}_{q}^n$ such that $x$ and $y$ differ only in position $i$ and $f(x)\neq f(y)$.
Let $R_{d,q}$ be the maximum number of relevant variables that a Boolean function of degree $d$ on $\mathbb{Z}_q^n$ can have,
and let $C_{d,q}=R_{d,q}q^{-d}$.

In 1994, Nisan and Szegedy \cite{NS94} proved that $R_{d,2}\leq d\cdot 2^{d-1}$.
In 2020, Chiarelli, Hatami and Saks \cite{CHS20} improved this bound to $6.614\cdot 2^{d}$.
In 2022, Wellens \cite{W22} further improved it to $4.394\cdot 2^{d}$.
In \cite{FI19J}, Filmus and Ihringer proved an analogue of the Nisan–Szegedy theorem for Boolean functions of degree $d$ on the Johnson graph.

In this work, we consider Boolean functions of a bounded degree on $\mathbb{Z}_{q}^n$ for $q\geq 3$.
It is known that any Boolean function of degree $1$ on $\mathbb{Z}_{q}^n$ has at most one relevant variable
(see, for example, \cite{FI19}).
The remark of Filmus and Ihringer at the end of \cite{FI19J} and Wellens's result imply that $R_{d,q}\leq 4.394\cdot 2^{\lceil \log_2 q\rceil d}$.
For $q\neq 2^s$, this bound was recently improved by the author \cite{V26} to $\frac{dq^{d+1}}{4(q-1)}$.
Another similar bound was recently established by Potapov in \cite{P25} (it depends not only on $d$ and $q$).

In this work, we prove a $q$-ary analogue of the result of Chiarelli, Hatami and Saks for all $q\geq 3$.
Namely, we prove that $R_{d,q}\leq m_qq^d$, where $m_q=\frac{2q(q^2+4q+1)}{(q-1)^4}$, for all $q\geq 3$.
For $q\in \{3,4,5,6,7\}$, we improve the bound $R_{d,q}\leq m_qq^d$ as follows:

\begin{enumerate}
  \item $R_{d,3}\leq 2.854\cdot 3^d$.
  
  \item $R_{d,4}\leq 1.749\cdot 4^d$.
  
  \item $R_{d,5}\leq 1.263\cdot 5^d$.
  
  \item $R_{d,6}\leq 0.994\cdot 6^d$.
  
  \item $R_{d,7}\leq 0.814\cdot 7^d$.
\end{enumerate}

We also note that there is a trivial lower bound $R_{d,q}\geq \frac{q^d-1}{q-1}$.
Indeed, using a $q$-ary generalization of binary decision trees, 
it is easy to construct a Boolean function of degree $d$ with $\sum_{i=0}^{d-1}q^i$ relevant variables.
Therefore, we have $R_{d,q}\geq \frac{q^d-1}{q-1}$ (see also Proposition \ref{P:vozrastaet}).

The paper is organized as follows.
In Section \ref{Sec:Def}, we introduce basic definitions.
In Section \ref{Sec:h_d<}, we prove that $h_{d,q}\leq 2d^3$.
In Section \ref{Sec:C_d=W_d}, we prove that $C_{d,q}=W_{d,q}$.
In Section \ref{Sec:C_d<}, we prove that $C_{d,q}\leq C_{d-1,q}+h_{d,q}q^{-d}$.
In Section \ref{Sec:Main}, we prove that $R_{d,q}\leq m_qq^d$, where $m_q=\frac{2q(q^2+4q+1)}{(q-1)^4}$, for all $q\geq 3$.
For $q\in \{3,4,5,6,7\}$, we improve this bound to $2.854\cdot 3^d$, $1.749\cdot 4^d$, $1.263\cdot 5^d$, $0.994\cdot 6^d$, and $0.814\cdot 7^d$, respectively.

\section{Basic definitions}\label{Sec:Def}

For a positive integer $n$, denote $[n]=\{1,\ldots,n\}$.
The {\em support} of a tuple $u \in \mathbb{Z}_{q}^n$ is defined as $\mathrm{supp}(u)=\{i\in [n]:u_i \neq 0\}$.
The {\em weight} of a tuple $u\in \mathbb{Z}_{q}^n$, denoted by $|u|$, is the number of its non-zero coordinates.
The set of all real-valued (Boolean) functions on $\mathbb{Z}_{q}^n$ is denoted by $\mathcal{R}(n,q)$ ($\mathcal{B}(n,q)$).

Now we discuss three equivalent definitions of degree for real-valued functions on $\mathbb{Z}_{q}^n$.

\textbf{Definition 1.}
For every $u\in \mathbb{Z}_{q}^n$, we define a function $\chi_u$ on $\mathbb{Z}_{q}^n$ by $\chi_{u}(x)=\omega^{\langle u,x\rangle}$, where 
$\omega=e^{2pi/q}$ is a primitive $q$th root of unity and
$\langle u,x\rangle=u_1x_1+\cdots+u_nx_n$.
Recall that a character of an Abelian group $G$ is a homomorphism from $G$ to $\mathbb{C}^*$.
It is easy to verify that the functions $\chi_u$, where $u\in \mathbb{Z}_{q}^n$, are characters of the Abelian group $\mathbb{Z}_{q}^n$.
It is well known that the characters of any Abelian group form an orthogonal basis for the vector space of all complex-valued functions defined on this group.
Therefore, every function $f\in \mathcal{R}(n,q)$ is uniquely expressible as
$$f=\sum_{u\in \mathbb{Z}_{q}^n}\widehat{f}(u)\chi_u,$$
where $\widehat{f}(u)\in \mathbb{C}$ for all $u\in \mathbb{Z}_{q}^n$.
The complex numbers $\widehat{f}(u)$, where $u\in \mathbb{Z}_{q}^n$, are called the {\em Fourier coefficients} of $f$.
The {\em degree} of a function $f\in \mathcal{R}(n,q)$, denoted by $\mathrm{deg}(f)$, is $\max\{|u|:u\in \mathbb{Z}_q^n~\mathrm{and}~\widehat{f}(u)\neq 0\}$.

\textbf{Definition 2.}
A function $f\in \mathcal{R}(n,q)$ is called a {\em $k$-junta} if it has at most $k$ relevant variables.
The {\em degree} of a function $f\in \mathcal{R}(n,q)$ is the smallest non-negative integer $t$ such that
$f$ is a sum of $t$-juntas.

\textbf{Definition 3.}
For $u\in \mathbb{Z}_{q}^n$, let $x^u=x_1^{u_1}\cdots x_n^{u_n}$.
It is known that for any function $f\in \mathcal{R}(n,q)$ there exists a unique polynomial $P\in \mathbb{R}[x_1,\ldots,x_n]$
of degree at most $q-1$ in each variable such that $f(x)=P(x)$ for all  $x\in \mathbb{Z}_{q}^n$
(see, for example, \cite[Proposition 2.2]{AG-MK24} and \cite[Proposition 1]{P25}).
Therefore, any function $f\in \mathcal{R}(n,q)$ can be uniquely represented as
$$f=\sum_{u\in \mathbb{Z}_q^n}\alpha_u x^u,$$ where $\alpha_u$ are real numbers.
The {\em degree} of a function $f\in \mathcal{R}(n,q)$ is $\max\{|u|:u\in \mathbb{Z}_q^n~\mathrm{and}~\alpha_u\neq 0\}$.

Let us show that these three definitions are equivalent.
The equivalence of the first and second definitions follows directly from the following two facts:

\begin{itemize}
  
  \item For every $u\in \mathbb{Z}_{q}^n$, $\chi_u$ is a $|u|$-junta. 
  
  \item Any $k$-junta is a linear combination of characters $\chi_u$ such that $|u| \leq k$.
\end{itemize}

Similarly, we can prove that the second and third definitions are equivalent.
Thus, all three definitions are equivalent.
In the rest of the paper we will use the third definition.

Let $\mathcal{B}(n,q,\leq d)$ denote the set of all Boolean functions of degree at most $d$ on $\mathbb{Z}_{q}^n$.

\textbf{Relevant variables.}
Let $f\in \mathcal{R}(n,q)$.
A variable $x_i$, $i\in [n]$, is called {\em relevant} for $f$
if there are two tuples $x,y\in \mathbb{Z}_{q}^n$ such that $x$ and $y$ differ only in position $i$ and $f(x)\neq f(y)$.
The number of relevant variables of $f$ is denoted by $R(f)$.
Let $R_{d,q}$ denote the maximum of $R(f)$ over all Boolean functions $f$ of degree at most $d$,
and let $C_{d,q}=R_{d,q}q^{-d}$.

\textbf{Maxonomial and hitting sets.}
Let $f\in \mathcal{B}(n,q)$. A set $S\subseteq [n]$ of size $\mathrm{deg}(f)$ is called {\em maxonomial} for $f$ 
if there exists $u\in \mathbb{Z}_{q}^n$ such that $\mathrm{supp}(u)=S$ and $\alpha_u\neq 0$.
A set $H \subseteq [n]$ is called {\em maxonomial hitting} for $f$ if $H$ intersects every maxonomial set of $f$.
Let $h(f)$ denote the minimum size of a maxonomial hitting set for $f$,
and let $h_{d,q}$ denote the maximum of $h(f)$ over all Boolean functions of degree $d$.

Let us consider $m$ disjoint subsets $B_1,\ldots,B_m$ of $[n]$ of sizes $k_1,\ldots,k_m$, respectively.
For each $s\in [m]$, order the elements of $B_s$ in increasing order: $B_s=\{b_{s,1},b_{s,2},\dots,b_{s,k_s}\}$, where $b_{s,1}<b_{s,2}<\dots<b_{s,k_s}$.
Let $v_1\in \mathbb{Z}_q^{k_1},v_2\in \mathbb{Z}_q^{k_2},\ldots,v_m\in \mathbb{Z}_q^{k_m}$.
For a tuple $x \in \mathbb{Z}_{q}^n$, we define the tuple $x^{(B_1,v_1),\ldots,(B_m,v_m)} \in \mathbb{Z}_{q}^n$ as follows:
$$
(x^{(B_1,v_1),\ldots,(B_m,v_m)})_i=
\begin{cases}
    x_i\oplus (v_s)_j, & \text{if } i=b_{s,j}\in B_s \text{ for some } s\in [m]; \\
    x_i, & \text{if } i \notin \bigcup_{s=1}^m B_s.
\end{cases}
,$$
where $\oplus$ denotes addition in $\mathbb{Z}_q$ and $(v_s)_j$ is the $j$-th entry of $v_s$.
For example, if $n=5$, $q=3$, $B_1=\{1,2\}$, $B_2=\{3,4\}$, $v_1=(1,1)$, $v_2=(2,2)$ and $x=(0,1,0,2,0)$, 
then $$x^{(B_1,v_1),(B_2,v_2)}=(1,2,2,1,0).$$

\textbf{Block sensitivity.}
Let $f\in \mathcal{B}(n,q)$.
For a tuple $x\in \mathbb{Z}_{q}^n$, the {\em local block sensitivity} of $f$ at $x$, denoted by $\mathrm{bs}(f,x)$, is the maximum number $t$ for which there exist $t$ disjoint subsets $B_1,\ldots,B_t$ of $[n]$ and $t$ tuples $v_1\in [q-1]^{|B_1|},\ldots,v_t\in [q-1]^{|B_t|}$ such that $f(x)\neq f(x^{(B_i,v_i)})$ for all $i\in [t]$.
The {\em block sensitivity} of $f$ is defined as $$\mathrm{bs}(f)=\max_{x \in \mathbb{Z}_{q}^n} \mathrm{bs}(f,x).$$

\textbf{Width of monomials.}
The {\em width} of a monomial $\alpha x^u$, where $\alpha\in \mathbb{R}$ and $u\in \mathbb{Z}_{q}^n$, is defined as $\vert{}u\vert{}$.
For example, the width of the monomial $5x_1x_2^2x_3^3$ is $3$.

\textbf{Weight of variables and functions.}
Let $f\in \mathcal{R}(n,q)$.
For a variable $x_i$, $i\in [n]$, let $\mathrm{deg}_i(f)=\max\{|u|:u\in \mathbb{Z}_q^n,i\in\mathrm{supp}(u)~\mathrm{and}~\alpha_u\neq 0\}$.
Let $w_i(f)=0$ if $x_i$ is not a relevant variable of $f$, and let $w_i(f)=q^{-\mathrm{deg}_i(f)}$ otherwise.
The {\em weight} of $f$, denoted by $W(f)$, is $\sum_{i\in [n]}w_i(f)$.
Let $W_{d,q}$ denote the maximum of $W(f)$ over all Boolean functions $f$ of degree at most $d$.
Note that $W_{d,q}$ is well-defined. 
Indeed, by \cite[Corollary 2]{V26}, we have $R_{d,q}\leq \frac{dq^{d+1}}{4(q-1)}$. 
Therefore, this maximum is taken over a finite set of functions.
A function $f\in \mathcal{B}(n,q,\leq d)$ for which $W(f)=W_{d,q}$ is called {\em $W_{d,q}$-maximizing}.

\section{Upper bound for $h_{d,q}$}\label{Sec:h_d<}

In this section, we prove that $h(f)\leq 2\mathrm{deg}(f)^3$.
As an immediate consequence, we obtain that $h_{d,q}\leq 2d^3$.

\begin{lemma}\label{L:m-block}
Let $f\in \mathcal{B}(n,q)$ and let $\mathrm{deg}(f)\geq 1$.
For any maxonomial set $M$ of $f$, there exist a set $B\subseteq M$ and a tuple $v\in [q-1]^{|B|}$ such that $f(\bar{0})\neq f(\bar{0}^{(B,v)})$.
\end{lemma}
\begin{proof}
Since $M$ is a maxonomial set for $f$, there exists a tuple  $u\in \mathbb{Z}_{q}^n$ such that
$\mathrm{supp}(u)=M$ and $\alpha_u\neq 0$.
Let $g$ be the restriction of $f$ to the set $\{x\in \mathbb{Z}_q^n:x_i=0 \text{ for each } i\in [n]-M\}$.
Since $x^u$ appears in $g$ with non-zero coefficient $\alpha_u$, $g$ is not a constant.
Therefore, there exists a tuple $w$ such that $g(w)\neq g(\bar{0})$.
This implies that there exists a tuple $\widetilde{w}\in \mathbb{Z}_{q}^n$ such that $\mathrm{supp}(\widetilde{w})\subseteq M$ and $f(\widetilde{w})\neq f(\bar{0})$.
Then $f(\bar{0}^{(B,v)})\neq f(\bar{0})$,
where $B=\mathrm{supp}(\widetilde{w})$ and $v$ is the tuple obtained from $\widetilde{w}$ by deleting coordinates $x_i$ such that $i\not\in B$. 
\end{proof}

\begin{lemma}\label{L:h(f)<deg(f)bs(f)}
For any function $f\in \mathcal{B}(n,q)$, we have $h(f)\leq \mathrm{deg}(f) \mathrm{bs}(f)$.
\end{lemma}
\begin{proof}
We will construct a maxonomial hitting set using the following algorithm.
In the first step of the algorithm, we choose an arbitrary maxonomial set $M_1$ of $f$.
For $j \ge 2$, in the $j$-th step of the algorithm, we choose a maxonomial set $M_j$ of $f$ such that $M_j \cap M_i = \emptyset$ for all $i \in [j-1]$.
If we cannot choose such a set $M_j$, then the algorithm terminates.
Assume that this algorithm terminates after $k$ steps, that is, we have chosen $k$ maxonomial sets $M_1, \ldots, M_k$.
Let $H=M_1\cup \ldots \cup M_k$.
Note that $H$ intersects any maxonomial set of $f$.
Therefore, $H$ is a maxonomial hitting set for $f$.
Let us show that $|H|\leq \mathrm{deg}(f)\mathrm{bs}(f)$.
By Lemma \ref{L:m-block}, for every $i\in [k]$ there exist a set $B_i\subseteq M_i$ and a tuple $v_i\in [q-1]^{|B_i|}$ such that $f(\bar{0})\neq f(\bar{0}^{(B_i,v_i)})$.
Therefore, we have $\mathrm{bs}(f)\geq k$.
Since $M_i$ is a maxonomial set of $f$ for each $i\in [k]$, we have $|M_i|=\mathrm{deg}(f)$ for all $i\in [k]$.
Therefore, we have $$|H|=|M_1|+\ldots +|M_k|=k\cdot \mathrm{deg}(f)\leq \mathrm{bs}(f)\mathrm{deg}(f).$$
\end{proof}

For $q=2$, Lemmas \ref{L:m-block} and \ref{L:h(f)<deg(f)bs(f)} were proved by Nisan and Smolensky (see, for example, Lemmas 5 and 6 in \cite{BW02}).


It remains to prove that $\mathrm{bs}(f)\leq 2\mathrm{deg}(f)^2$.
For Boolean functions on $\mathbb{Z}_{2}^n$, this bound was proved by Nisan and Szegedy in \cite {NS94}.
In this work, we show that the same result holds for Boolean functions on $\mathbb{Z}_{q}^n$ for all $q\geq 3$.
We need the following useful fact.
\begin{lemma}[\cite{NS94}, Lemma 3.5]\label{L:deg(f)>n/2}
Let $f\in \mathcal{B}(n,2)$.
If $f(x)\neq f(\overline{0})$ for any vector $x\in \mathbb{Z}_{2}^n$ of weight $1$,
then $\mathrm{deg}(f)\geq \sqrt{n/2}$.
\end{lemma}

\begin{lemma}\label{L:deg(f)>bs(f)/2}
For any function $f\in \mathcal{B}(n,q)$, we have $\mathrm{deg}(f)\geq \sqrt{\mathrm{bs}(f)/2}$.
\end{lemma}
\begin{proof}
Let $t=\mathrm{bs}(f)$ and let $z\in \mathbb{Z}_{q}^n$ be a tuple such that $\mathrm{bs}(f,z)=t$.
By the definition of local block sensitivity, there exist $t$ disjoint subsets $B_1,\ldots,B_t$ of $[n]$ and $t$ tuples $v_1,\ldots,v_t$ 
such that $f(z)\neq f(z^{(B_i,v_i)})$ for all $i\in [t]$.
We define a function $g$ on $\mathbb{Z}_{2}^t$ as follows: 

$$g(y_1,\ldots,y_t)=f(z^{(B_1,y_1v_1),\ldots,(B_t,y_tv_t)}).$$

Note that $g$ is Boolean.
In addition, since $f(z)\neq f(z^{(B_i,v_i)})$ for all $i\in [t]$,
we have that $g(y)\neq g(\overline{0})$ for any vector $y\in \mathbb{Z}_{2}^t$ of weight $1$.
Thus, $g$ satisfies the conditions of Lemma \ref{L:deg(f)>n/2}.
Applying Lemma \ref{L:deg(f)>n/2} to $g$, we obtain that $\mathrm{deg}(g)\geq \sqrt{t/2}$.
On the other hand, it is easy to check that $\mathrm{deg}(g)\leq \mathrm{deg}(f)$.
Therefore, we have $\mathrm{deg}(f)\geq \sqrt{\mathrm{bs}(f)/2}$.
\end{proof}

Combining Lemmas \ref{L:h(f)<deg(f)bs(f)} and \ref{L:deg(f)>bs(f)/2}, we immediately obtain the following.

\begin{corollary}
For any function $f\in \mathcal{B}(n,q)$, we have $h(f)\leq 2\mathrm{deg}(f)^3$.
\end{corollary}

\begin{corollary}\label{Cor:h_d<}
For all $d\geq 1$, we have $h_{d,q}\leq 2d^3$.
\end{corollary}

\section{$C_{d,q}=W_{d,q}$}\label{Sec:C_d=W_d}
In this section, we prove that $C_{d,q}=W_{d,q}$.

\begin{lemma}\label{L:W_d=C_d}
Let $f\in \mathcal{B}(n,q,\leq d)$. If $f$ is $W_d$-maximizing, 
then every relevant variable of $f$ belongs to a monomial of width $d$.
\end{lemma}
\begin{proof}
Let $x_1,\ldots,x_n$ be the relevant variables of $f$.
Assume for contradiction that there are $\ell\geq 1$ variables that do not belong to any monomial of width $d$,
and that these variables are $x_1,\ldots,x_{\ell}$.
We will construct a function $g$ of degree at most $d$ such that $W(g)>W(f)$.

For $0\leq i\leq q-1$, we define a polynomial $P_i$ of degree $q-1$ as follows:

$$P_i(x)=\prod_{j=0,j\neq i}^{q-1}\frac{x-j}{i-j}$$

The polynomials $P_0,\ldots,P_{q-1}$ have the following properties:

(S1) For every $i\in \{0,1,\ldots,q-1\}$, we have $P_i(j)=0$ for all $j\in  \{0,1,\ldots,q-1\}\setminus \{i\}$ and $P_i(i)=1$.

(S2) The polynomial $\sum_{i=0}^{q-1}P_i$ is identically equal to one, that is, this polynomial is a constant.

For $0\leq i\leq q-1$, we define a function $f_i$ by the following rule:
$$
f_i(x_1,x_2,\ldots,x_{n+(q-1)\ell})=\begin{cases}
f(x_1,\ldots,x_n),&\text{if $i=0$;}\\
f(x_{n+(i-1)\ell+1},\ldots,x_{n+i\ell},x_{\ell+1},\ldots,x_n),&\text{if $i\in [q-1]$.}
\end{cases}
$$

Note that all functions $f_0,f_1,\ldots,f_{q-1}$ are Boolean.
We define a function $g$ as follows:

$$g(x_1,x_2,\ldots,x_{n+(q-1)\ell+1})=\sum_{i=0}^{q-1}P_i(x_{n+(q-1)\ell+1})f_i$$

Using (S1), we obtain that $g$ is equal to $f_t$ if $x_{n+(q-1)\ell+1}=t$, where $t\in \mathbb{Z}_q$.
Therefore, $g$ is Boolean.

Let us show that $\mathrm{deg}_i(g)=\mathrm{deg}_i(f)$ for all $i\in \{\ell+1,\ldots,n\}$.
Since $x_i$ does not appear in monomials of $f$ of width $d$ for any $i\leq \ell$, 
the functions $f_0,\ldots,f_{q-1}$ have the same set of monomials of width $d$.
Using this observation and (S2), we see that any monomial of $f$ of width $d$ is a monomial of $g$
(with the same coefficient).
In addition, if $m$ is a monomial of $f_j$ of width at most $d-1$, then $m$ gives monomials of width at most $d$ in $P_jf_j$.
Hence, for every $0\leq j\leq q-1$ any monomial of $f_j$ of width at most $d-1$ gives monomials of width at most $d$ in $g$.
Therefore, we have $\mathrm{deg}_i(g)=\mathrm{deg}_i(f)$ for all $i\in \{\ell+1,\ldots,n\}$.
Thus, $w_i(g)=w_i(f)$ for all $i\in \{\ell+1,\ldots,n\}$.


Let us show that $\mathrm{deg}_{n+(i-1)\ell+r}(g)=\mathrm{deg}_{r}(g)=\mathrm{deg}_{r}(f)+1$ for all $i\in [q-1]$ and $r\in [\ell]$.
Let us fix $i\in [q-1]$ and $r\in [\ell]$.
Note that $$\mathrm{deg}_{n+(i-1)\ell+r}(g)=\mathrm{deg}_{n+(i-1)\ell+r}(P_i(x_{n+(q-1)\ell+1})f_i)=\mathrm{deg}_{n+(i-1)\ell+r}(f_i)+1$$ 
and
$$\mathrm{deg}_{r}(g)=\mathrm{deg}_{r}(P_0(x_{n+(q-1)\ell+1})f_0)=\mathrm{deg}_{r}(f_0)+1=\mathrm{deg}_{r}(f)+1.$$
Therefore, it is sufficient to prove that $\mathrm{deg}_{n+(i-1)\ell+r}(f_i)=\mathrm{deg}_{r}(f_0)$.
The last equality follows from the fact that the set of monomials of $f_i$ containing $x_{n+(i-1)\ell+r}$ is obtained from the set of monomials of $f_0$ containing $x_r$ 
by replacing the variables $x_1,\ldots,x_{\ell}$ with the variables $x_{n+(i-1)\ell+1},\ldots,x_{n+i\ell}$.
Therefore, we have  $\mathrm{deg}_{n+(i-1)\ell+r}(g)=\mathrm{deg}_{r}(g)=\mathrm{deg}_{r}(f)+1$.
This immediately implies that $w_{n+(i-1)\ell+r}(g)=w_r(g)=\frac{1}{q}w_r(f)$ for all $i\in [q-1]$ and $r\in [\ell]$.
Thus, we have

\begin{align*}
W(g) & =\sum_{i=1}^{n+(q-1)\ell+1}w_i(g)=\sum_{i=1}^{\ell}w_i(g)+\sum_{i=\ell+1}^{n}w_i(g)+\sum_{i=n+1}^{n+(q-1)\ell}w_i(g)+w_{n+(q-1)\ell+1}(g)= \\
     & =\frac{1}{q}\sum_{i=1}^{\ell}w_i(f)+\sum_{i=\ell+1}^{n}w_i(f)+\sum_{j=1}^{q-1}\sum_{r=1}^{\ell}w_{n+(j-1)\ell+r}(g)+w_{n+(q-1)\ell+1}(g)= \\
     & =\frac{1}{q}\sum_{i=1}^{\ell}w_i(f)+\sum_{i=\ell+1}^{n}w_i(f)+\frac{1}{q}\sum_{j=1}^{q-1}\sum_{r=1}^{\ell}w_r(f)+w_{n+(q-1)\ell+1}(g)= \\
     & =\frac{1}{q}\sum_{i=1}^{\ell}w_i(f)+\sum_{i=\ell+1}^{n}w_i(f)+\frac{q-1}{q}\sum_{r=1}^{\ell}w_r(f)+w_{n+(q-1)\ell+1}(g)= \\
     & =\sum_{i=1}^{\ell}w_i(f)+\sum_{i=\ell+1}^{n}w_i(f)+w_{n+(q-1)\ell+1}(g)=W(f)+w_{n+(q-1)\ell+1}(g).
\end{align*}

Note that $x_{n+(q-1)\ell+1}$ is a relevant variable of $g$.
Indeed, for any monomial $m$ of $f$ containing $x_1$, $\frac{(-1)^{q-1}}{(q-1)!}mx_{n+(q-1)\ell+1}^{q-1}$ is a monomial of $g$.
Therefore, we have $W(g)>W(f)$.
Thus, $g$ is a Boolean function of degree at most $d$ with $W(g)>W(f)$.
This gives us the required contradiction.

\end{proof}

\begin{corollary}\label{Cor:C_d=W_d}
For all $d\geq 1$, we have $C_{d,q}=W_{d,q}$.
\end{corollary}
\begin{proof}
Let $f\in \mathcal{B}(n,q,\leq d)$ and let $f$ have $R_{d,q}$ relevant variables.
We have
$$W(f)=\sum_{i\in [n]}w_i(f)\geq R(f)q^{-d}=R_{d,q}q^{-d}=C_{d,q}.$$
Therefore, $C_{d,q}\leq W_{d,q}$.

Let $f\in \mathcal{B}(n,q,\leq d)$ and let $f$ be $W_d$-maximizing.
By Lemma \ref{L:W_d=C_d}, we have
$$W(f)=\sum_{i\in [n]}w_i(f)=R(f)q^{-d}\leq R_{d,q}q^{-d}=C_{d,q}.$$
Since $W(f)=W_{d,q}$, we have $C_{d,q}\geq W_{d,q}$.
\end{proof}

\section{Upper bounds for $W_{d,q}$ and $C_{d,q}$}\label{Sec:C_d<}

In this section, we prove that $W_{d,q}\leq W_{d-1,q}+h_{d,q}q^{-d}$.
As an immediate consequence, we obtain that $C_{d,q}\leq C_{d-1,q}+h_{d,q}q^{-d}$.

We recall several definitions.
A {\em partial assignment} is a mapping $\alpha:[n]\rightarrow \{0,1,\ldots,q-1,*\}$,
and $\mathrm{Fixed}(\alpha)$ is the set $\{i\in [n]:\alpha(i)\in \{0,1,\ldots,q-1\}\}$.
For $J\subseteq [n]$, let $\mathrm{PA}(J)$ be the set of all partial assignments $\alpha$ with $\mathrm{Fixed}(\alpha)=J$.
The {\em restriction} of a function $f\in \mathcal{R}(n,q)$ by $\alpha$, denoted by $f_{\alpha}$, is the function on variable set $\{x_i:i\in [n]-\mathrm{Fixed}(\alpha)\}$
obtained by setting $x_i=\alpha_i$ for each $i\in \mathrm{Fixed}(\alpha)$.

The following result is a $q$-ary generalization of Claim 2.4 from \cite{CHS20}.

\begin{lemma}\label{L:w_i(f)<}
Let $f\in \mathcal{R}(n,q)$.
For every $J\subseteq [n]$ and $i\not\in J$, we have 
$$w_i(f)\leq q^{-|J|}\sum_{\alpha\in \mathrm{PA}(J)}w_i(f_{\alpha}).$$
\end{lemma}
\begin{proof}
If $x_i$ is not a relevant variable of $f$, then $w_i(f)=0$ and the statement of the lemma holds automatically.
Thus, we can assume that $x_i$ is a relevant variable of $f$.
Fix $j\in J$.
For $a\in \mathbb{Z}_q$, let $f_a$ be the restriction of $f$ to the set $\{x\in \mathbb{Z}_q^n:x_j=a\}$.
Let us prove the statement of the lemma by induction on $|J|$.
First, we consider the case $|J|=1$.
There are two cases for monomials of $f$ of width $\mathrm{deg}_i(f)$ containing $x_i$.

Suppose that there exists a monomial of $f$ of width $\mathrm{deg}_i(f)$ containing $x_i$ and not containing $x_j$.
Note that $f_a$ contains this monomial for each $a\in \mathbb{Z}_q$.
Therefore, we have $\mathrm{deg}_i(f_a)=\mathrm{deg}_i(f)$ for all $a\in \mathbb{Z}_q$.
Thus, in this case we have $w_i(f)=\frac{1}{q}\sum_{a=0}^{q-1}w_i(f_a)$.

Suppose that all monomials of $f$ of width $\mathrm{deg}_i(f)$ containing $x_i$ also contain $x_j$.
In this case, we have $\mathrm{deg}_i(f_a)\leq \mathrm{deg}_i(f)-1$ for all $a\in \mathbb{Z}_q$.
Since $x_i$ is a relevant variable for $f$, there exists $b\in \mathbb{Z}_q$ such that $x_i$ is a relevant variable for $f_b$.
Therefore, we have $w_i(f)\leq \frac{1}{q}w_i(f_b)$.
This implies that $w_i(f)\leq \frac{1}{q}\sum_{a=0}^{q-1}w_i(f_a)$.

Let us prove the induction step for $|J|\geq 2$.
Using the induction hypothesis for $f_a$, where $a\in \mathbb{Z}_q$, and the set $J-\{j\}$, 
we obtain that the inequality 

\begin{equation}\label{Eq:1}
w_i(f_a)\leq q^{1-|J|}\sum_{\beta\in \mathrm{PA}(J-\{j\})}w_i((f_a)_{\beta})
\end{equation}
holds for all  $a\in \mathbb{Z}_q$.
Combining (\ref{Eq:1}) and the inequality $w_i(f)\leq \frac{1}{q}\sum_{a=0}^{q-1}w_i(f_a)$, we get

$$w_i(f)\leq q^{-|J|}\sum_{\alpha\in \mathrm{PA}(J)}w_i(f_{\alpha}).$$
\end{proof}

The following result is a $q$-ary generalization of Proposition 2.3 from \cite{CHS20}.

\begin{lemma}\label{L:W_d<}
For all $d\geq 1$, we have $W_{d,q}\leq W_{d-1,q}+h_{d,q}q^{-d}$.
\end{lemma}
\begin{proof}
Let $f\in \mathcal{B}(n,q,\leq d)$ and let $f$ be $W_d$-maximizing.
Let $H$ be a maxonomial hitting set for $f$ of minimum size. 
Note that $\mathrm{deg}_i(f)=d$ for all $i\in H$,
since otherwise $H-\{i\}$ would be a smaller maxonomial hitting set.
Thus, we have:

\begin{equation}\label{Eq:2}
W(f)=\sum_{i\in [n]}w_i(f)=q^{-d}|H|+\sum_{i\in [n]-H}w_i(f).
\end{equation}

Using Lemma \ref{L:w_i(f)<} for $J=H$ and summing over $i\in [n]-H$, 
we get the following:
\begin{equation}\label{Eq:3}
\begin{split}
& \sum_{i\in [n]-H}w_i(f)\leq q^{-|H|}\sum_{i\in [n]-H}\sum_{\alpha\in \mathrm{PA}(H)}w_i(f_{\alpha})= \\
& =q^{-|H|}\sum_{\alpha\in \mathrm{PA}(H)}W(f_{\alpha})\leq W_{d-1,q},
\end{split}
\end{equation}
where the last inequality follows since $\mathrm{deg}(f_{\alpha})\leq d-1$ for all $\alpha\in \mathrm{PA}(H)$.
Combining (\ref{Eq:2}) and (\ref{Eq:3}), we immediately obtain that $W_{d,q}\leq W_{d-1,q}+h_{d,q}q^{-d}$.
\end{proof}

Using Lemma \ref{L:W_d<} and Corollary \ref{Cor:C_d=W_d}, we immediately obtain the following.

\begin{corollary}\label{Cor:C_d<}
For all $d\geq 1$, we have $C_{d,q}\leq C_{d-1,q}+h_{d,q}q^{-d}$.  
\end{corollary}

\section{Main results}\label{Sec:Main}
The main result of this paper is the following.
\begin{theorem}\label{T:1}
For every $q\geq 3$, we have $R_{d,q}\leq m_qq^d$, where $m_q=\frac{2q(q^2 + 4q + 1)}{(q-1)^4}$.
\end{theorem}
\begin{proof}
By Corollary \ref{Cor:C_d<}, we have $C_{d,q}\leq C_{d-1,q}+h_{d,q}q^{-d}$.
This implies that $C_{d,q}\leq \sum_{i=1}^{d}h_{i,q}q^{-i}$.
Combining this inequality with Corollary \ref{Cor:h_d<}, we obtain that
$C_{d,q}\leq 2\sum_{i=1}^{d}i^3q^{-i}$.
Therefore, $C_{d,q}\leq 2\sum_{i=1}^{\infty}i^3q^{-i}$.
Since $\sum_{i=1}^{\infty}i^3q^{-i}=\frac{q(q^2 + 4q + 1)}{(q-1)^4}$, we have $C_{d,q}\leq m_q$.
\end{proof}

For $q\in \{3,4,5,6,7\}$, we can improve the bound $R_{d,q} \le m_q q^d$ as follows.
First, we need the following fact.

\begin{proposition}\label{P:vozrastaet}
For all $d\geq 1$, we have $C_{d,q}\geq C_{d-1,q}+q^{-d}$.
\end{proposition}
\begin{proof}
Let $f\in \mathcal{B}(n,q,\leq d-1)$ and let $R(f)=R_{d-1,q}$.
For $0\leq i\leq q-1$, we define a function $f_i$ on $\mathbb{Z}_{q}^{qn}$ as follows: 
$$f_i(x_1,x_2,\ldots,x_{qn})=f(x_{in+1},\ldots,x_{(i+1)n}).$$

For $0\leq i\leq q-1$, we define a polynomial $P_i$ of degree $q-1$ as follows:
$$P_i(y)=\prod_{j=0,j\neq i}^{q-1}\frac{y-j}{i-j}$$

We define a function $g$ on $\mathbb{Z}_{q}^{qn+1}$ as follows: 
$$g(x_1,x_2,\ldots,x_{qn},y)=\sum_{i=0}^{q-1}P_i(y)f_i.$$

Note that $g$ is Boolean.
We also note that $\mathrm{deg}(g)=\mathrm{deg}(f)+1$.
Therefore, $\mathrm{deg}(g)\leq d$.
Moreover, we have 
$$R(g)=\sum_{i=0}^{q-1}R(f_i)+1=qR(f)+1=qR_{d-1,q}+1.$$
Hence, $R_{d,q}\geq qR_{d-1,q}+1$.
This implies that $C_{d,q}\geq C_{d-1,q}+q^{-d}$.
\end{proof}

Proposition \ref{P:vozrastaet} implies that the sequence $C_{d,q}$ is increasing in $d$.
In addition, by Theorem \ref{T:1}, the sequence $C_{d,q}$ is bounded.
Therefore, the limit $\lim_{d \to \infty}C_{d,q}$ exists and is finite.
We denote this limit by $t_q$.

By Corollary \ref{Cor:C_d<}, we have $C_{d,q}\leq C_{d-1,q}+h_{d,q}q^{-d}$.
This implies that $C_{m,q}\leq C_{d,q}+\sum_{i=d+1}^{m}h_{i,q}q^{-i}$ for all $m>d$.
Hence, $t_q\leq C_{d,q}+\sum_{i=d+1}^{\infty}h_{i,q}q^{-i}$.
By \cite[Corollary 2]{V26}, we have $C_{d,q}\leq \frac{dq}{4(q-1)}$. 
Therefore, $t_q\leq \frac{dq}{4(q-1)}+\sum_{i=d+1}^{\infty}h_{i,q}q^{-i}$.
Combining this inequality with Corollary \ref{Cor:h_d<}, we get
$t_q\leq \frac{dq}{4(q-1)}+2\sum_{i=d+1}^{\infty}i^3q^{-i}$.
Since $\sum_{i=1}^{\infty}i^3q^{-i}=\frac{q(q^2 + 4q + 1)}{(q-1)^4}$, we have
$t_q\leq m_q+\frac{dq}{4(q-1)}-2\sum_{i=1}^{d}i^3q^{-i}$.
Thus,

\begin{equation}\label{Eq:4}
t_q\leq \min_{d\geq 0}\left(m_q+\frac{dq}{4(q-1)}-2\sum_{i=1}^{d}i^3q^{-i}\right)
\end{equation}

We have the following:

\begin{enumerate}
  
  \item For $q=3$, the minimum in \eqref{Eq:4} is achieved at $d=6$ and equals $\frac{1387}{486}$.
  Therefore, $t_3\leq 2.854$.
  
  \item For $q=4$, the minimum in \eqref{Eq:4} is achieved at $d=4$ and equals $\frac{1511}{864}$.
  Therefore, $t_4\leq 1.749$.
  
  \item For $q=5$, the minimum in \eqref{Eq:4} is achieved at $d=3$ and equals $\frac{10099}{8000}$.
  Therefore, $t_5\leq 1.263$.
  
  \item For $q=6$, the minimum in \eqref{Eq:4} is achieved at $d=2$ and equals $\frac{5588}{5625}$.
  Therefore, $t_6\leq 0.994$.
  
  \item For $q=7$, the minimum in \eqref{Eq:4} is achieved at $d=2$ and equals $\frac{2153}{2646}$.
  Therefore, $t_7\leq 0.814$.
  
  \item For $q\geq 8$, the minimum in \eqref{Eq:4} is achieved at $d=0$ and equals $m_q$.
  
\end{enumerate}

Taking into account that $C_{d,q}\le t_q$, we immediately obtain the following result.

\begin{theorem}
The following statements hold:

\begin{enumerate}
  
  \item For all $d\geq 1$, we have $R_{d,3}\leq 2.854\cdot 3^d$.
  
  \item For all $d\geq 1$, we have $R_{d,4}\leq 1.749\cdot 4^d$.
  
  \item For all $d\geq 1$, we have $R_{d,5}\leq 1.263\cdot 5^d$.
  
  \item For all $d\geq 1$, we have $R_{d,6}\leq 0.994\cdot 6^d$.
  
  \item For all $d\geq 1$, we have $R_{d,7}\leq 0.814\cdot 7^d$.

\end{enumerate}

\end{theorem}


\begin{thebibliography}{99}


\bibitem{AG-MK24}
S. Asensio, I. García-Marco, K. Knauer, Sensitivity of $m$-ary functions and low degree partitions of Hamming graphs,
arXiv:2409.16141, September 2024. 

\bibitem{BW02}
H. Buhrman, R. De Wolf, Complexity measures and decision tree complexity: a survey,
Theoretical Computer Science 288 (2002) 21--43.

\bibitem{CHS20}
J. Chiarelli, P. Hatami, M. Saks, An asymptotically tight bound on the number of relevant variables in a bounded degree Boolean function,
Combinatorica 40 (2020) 237--244.

\bibitem{FI19J}
Y. Filmus, F. Ihringer, Boolean constant degree functions on the slice are juntas, Discrete Mathematics 342(12) (2019) 111614.


\bibitem{FI19}
Y. Filmus, F. Ihringer, Boolean degree 1 functions on some classical association schemes, Journal of Combinatorial Theory, Series A 162 (2019) 241--270.


\bibitem{NS94}
N. Nisan, M. Szegedy, On the degree of Boolean functions as real polynomials, Computational Complexity 4 (1994) 301--313.


\bibitem{P25}
V. N. Potapov, On the number of relevant variables for discrete functions, Cryptography and Communications 17 (2025) 989--998.


\bibitem{V26}
A. Valyuzhenich, An upper bound on the number of relevant variables for Boolean functions on the Hamming graph,
Discrete Mathematics 349(2) (2026) 114745.

\bibitem{W22}
J. Wellens, Relationships between the number of inputs and other complexity measures of Boolean functions,
Discrete Analysis 20 (2022).


\end{thebibliography}
\end{document}